\documentclass[12pt]{amsart}
\usepackage{amsthm,amsmath,amsfonts,amssymb,amscd,mathrsfs,hyperref,cleveref}
\usepackage{latexsym,mathabx}
\usepackage{amsaddr}
\usepackage{cite}

\begingroup
\makeatletter
   \@for\theoremstyle:=definition,remark,plain\do{%
     \expandafter\g@addto@macro\csname th@\theoremstyle\endcsname{%
        \addtolength\thm@preskip\parskip
     }%
   }
\endgroup

\theoremstyle{plain}
\newtheorem{thm}{Theorem}[section]
\newtheorem{lem}{Lemma}
\newtheorem*{lem*}{Lemma}

\theoremstyle{definition}
\newtheorem{dfn}{Definition}

\crefname{thm}{theorem}{theorems}
\crefname{lem}{lemma}{lemmas}

\newcommand{\bs}{\boldsymbol}

\newcommand{\mc}{\mathcal}
\renewcommand{\mod}{\operatorname{mod}}

\numberwithin{equation}{section}

\def \a{\alpha} \def \b{\beta} \def \d{\delta} \def \e{\varepsilon} \def \g{\gamma}  \def \l{\lambda} \def \s{\sigma}  

\renewcommand{\labelenumi}{\setlength{\labelwidth}{\leftmargin}
   \addtolength{\labelwidth}{-\labelsep}
   \hbox to \labelwidth{\theenumi.\hfill}}

\newcommand{\twosum}[2]{\sum_{\substack{#1\\#2}}}
\newcommand{\threesum}[3]{\sum_{\substack{#1\\#2\\#3}}}
\newcommand{\R}{\mathbb R}
\renewcommand{\lessapprox}{\llcurly}
\renewcommand{\gtrapprox}{\ggcurly}
\begin{document}
\title{Bounded intervals containing many primes}
\author{R.C. Baker and A.J. Irving}

\maketitle

\begin{abstract}
By combining a sieve method of Harman with the work of Maynard and Tao we show that 
$$\liminf_{n\rightarrow \infty}(p_{n+m}-p_n)\ll \exp(3.815m).$$
\end{abstract}

\section{Introduction}

For any natural number $m\geq 1$ let
$$H_m=\liminf_{n\rightarrow \infty}(p_{n+m}-p_n),$$ 
where $p_n$ denotes the $n^{\mathrm{th}}$ prime.  It was shown by Maynard \cite{maynardsmall} that $H_m$ is finite for all $m$ and that it satisfies 
$$H_m\ll m^3 \exp(4m).$$
This bound was improved by Polymath \cite[Theorem 1.4(vi)]{polymath8b} to 
\begin{equation}\label{polybound}
H_m\ll m\exp\left(\left(4-\frac{28}{157}\right)m\right).
\end{equation}
The smaller exponent, $4-\frac{28}{157}=3.821\ldots$, was obtained by an application of a result of Polymath \cite{polymath8a} which extends the Bombieri-Vinogradov theorem to smooth moduli which are slightly larger than $x^{\frac12}$.  This latter result was an improvement of an earlier theorem of Zhang \cite{zhang}, who used it to obtain the first proof that $H_1$ is finite.  In the present work we combine the methods of Maynard and Polymath with a sieve procedure developed by Harman \cite{har83, har96, harmanbook} to obtain the following small improvement on \eqref{polybound}.

\begin{thm}\label{mainthm}
We have 
$$H_m\ll \exp(3.815m).$$
\end{thm}

This theorem will be proved by constructing a minorant for the indicator function of the primes for which we can prove a slightly stronger equidistribution theorem in arithmetic progressions to smooth moduli.  Given such a minorant, the following Lemma shows how Theorem \ref{mainthm} may be deduced.

\begin{lem}\label{sievelem}
Suppose that for all large $x$ there exists a function $\rho(n):[x,2x]\rightarrow \R$ satisfying the following properties:

\begin{enumerate}
\item $\rho(n)$ is a minorant for the indicator function of the primes, that is 
$$\rho(n)\leq \begin{cases}
1 & n\text{ is prime}\\
0 & \text{otherwise.}\\
\end{cases}$$

\item If $\rho(n)\ne 0$ then all prime factors of $n$ exceed $x^\xi$, for some fixed $\xi>0$.

\item The function $\rho(n)$ has exponent of distribution $\theta$ to smooth moduli (see Definition \ref{smoothlevel} below).   

\item We have 
$$\sum_{x\leq n\leq 2x}\rho(n)=(1-c_1+o(1))\frac{x}{\log x}.$$
\end{enumerate}
We then have 
$$H_m\ll \exp(c_0m)$$
for any 
$$c_0>
\frac{2m}{\theta(1-c_1)}.$$
The implied constant may depend on $c_0$.  
\end{lem}

We note that this Lemma implies \eqref{polybound} by taking $\rho(n)$ to be the indicator function of the primes, for which we can take $\theta=\frac{1}{2}+\frac{7}{300}$ by Polymath \cite[Theorem 2.4(i)]{polymath8a}.  Our improvement, Theorem \ref{mainthm}, will follow by taking the minorant constructed in the following result.

\begin{lem}\label{minorant}
Fix an $\eta$ in $(0,\frac{22}{3295})$.  Let $E(\eta)$ be the polyhedron in $\R^4$ given by
 \begin{align}\label{eq1.2}
 \begin{split}
\frac25+\eta>\a_1 &> \a_2 > \a_3>\a_4>\frac15-2\eta\\
\a_1 &+ \a_2 + \a_3 + 2\a_4 < 1\\
\a_1 &+ \a_2 < \frac 25 + \eta\\
\a_2 &+ \a_3 + \a_4 > \frac 35 - \eta
 \end{split} 
 \end{align}
and let $f(\bs \a) = \frac 1{\a_1\a_2\a_3\a_4(1-\a_1-\a_2-\a_3-
\a_4)}$.  Then, for all large $x$ there exists a function $\rho$ satisfying the hypotheses of the previous lemma with 
$$\theta=\frac12+\frac{7}{300}+\frac{17\eta}{120}$$
and 
\begin{equation}\label{c1def}
c_1=6\int_{E(\eta)} f(\bs \a)\,d\bs \a,
\end{equation}
where $d\bs \a = d\a_1 \ldots d\a_4$.
\end{lem}

We note that, for small $\eta$, the value given by \eqref{c1def} is $O(\eta^4)$ and therefore the quantity 
$$\frac{2}{(1-c_1(\eta))\theta(\eta)}$$
is decreasing for sufficiently small $\eta$. This shows that an improvement to \eqref{polybound} is possible by an appropriate choice of $\eta$.  

We now give the definition of ``exponent of distribution to smooth moduli'' which is needed for Lemma \ref{sievelem} and which will be used repeatedly throughout this work.  It is a generalisation to arbitrary functions of \cite[Claim 2.3]{polymath8a}.

\begin{dfn}\label{smoothlevel}
An arithmetic function $f$ with support contained in $[x,2x]$ has exponent of distribution $\theta$ to smooth moduli if for every $\epsilon>0$ there exists a $\delta>0$ for which the following holds.  

For any $P$ which is a product of distinct primes smaller than $x^\delta$, any integer $a$ with $(a,P)=1$ and any $A>0$ we have 
$$\twosum{q\leq x^{\theta-\epsilon}}{q|P}\left|\sum_{n\equiv a\pmod q}f(n)-\frac{1}{\phi(q)}\sum_{(n,q)=1}f(n)\right|\ll x(\log x)^{-A}.$$
The implied constant may depend on $\epsilon$ and $A$.
\end{dfn}

The restriction $q|P$ implies that $q$ runs over squarefree, $x^\delta$-smooth moduli.  We observe that if a bounded number of functions $f_i$ have exponent $\theta$ then so does their sum.  This also holds for the sum of at most $\log^{O(1)}x$ functions provided we take care that for each $\epsilon>0$ the resulting $\delta$ are bounded away from $0$.  

For the remainder of the paper we define 
\begin{equation}\label{theta0def}
\theta_0(\eta)=\frac{1}{2}+\frac{7}{300}+\frac{17\eta}{120}.
\end{equation}
Lemma \ref{minorant} will be proven by using a Buchstab decomposition to write the indicator function of the primes as a sum of various functions.  We will then show that each summand either has exponent $\theta_0(\eta)$ or it is positive. The minorant will be constructed by removing the latter summands.

\section{Arithmetic Information}

The Harman sieve will be used to decompose various functions as a sum of convolutions.  It will be shown that the latter have exponent $\theta_0(\eta)$ by using the results of Polymath \cite[Theorem 2.8]{polymath8a} as well as a new result, Lemma \ref{lem1}, with which we begin. 

In the following lemma the symbols $\lessapprox$ and $\gtrapprox$ have  the same meaning as in \cite{polymath8a}, that is $X\lessapprox Y$ means that 
$$X\leq x^{o(1)}Y.$$

 \begin{lem}\label{lem1}
Fix $\varpi,\delta,\sigma>0$.  Let $\b(n)$ be a smooth coefficient sequence at scale $N$ (see \cite[Definition 2.5]{polymath8a}) with 
\begin{equation}\label{lem1cond1}
x^{1/2-\sigma}\lessapprox N\lessapprox x^{1/2}
\end{equation}
and let $\a(m)$ be a coefficient sequence at scale $M$ with $MN\asymp x$.  Let $P$ be a product of distinct primes less than $x^\d$ and $a$ an integer with $(a,P)=1$.  We then have 
$$\twosum{d\leq x^{1/2+2\varpi}}{d|P}|\Delta(\a\star\b;a(d))|\ll x\log^{-A},$$
for any $A>0$, provided that 
\begin{equation}\label{lem1cond2}
4\sigma+28\varpi+8\delta<1.
\end{equation}
The discrepancy $\Delta$ is defined in \cite[(1.1)]{polymath8a}.  
\end{lem}
 
 \begin{proof}
We can follow the reduction to exponential sums in \cite[Section 5.3]{polymath8a} until we reach (5.28), with $\Phi_\ell$ defined by (5.29).  The only detail to check prior to that point is the constraint (5.15), which is used by Polymath to control various diagonal contributions.  Using (5.12) and (5.13) we obtain 
$$RQ^2=R^{-1}(RQ)^2 \lessapprox N^{-1}x^{1+4\varpi+\delta+3\e}$$
so that, by \eqref{lem1cond1}, 
$$RQ^2\lessapprox x^{1/2+\sigma+4\varpi+\delta+3\e}.$$
Equation (5.15) is therefore satisfied provided we take $\e$ sufficiently small and 
$$2\sigma+8\varpi+2\delta<1.$$
This constraint is weaker than our assumption \eqref{lem1cond2}.  

Just as in the proof of  \cite[Theorem 5.8]{polymath8a} it now suffices to prove (5.31). We treat the variable $r$ trivially. Thus, as in \cite[(5.32)]{polymath8a},  we shall show that
 \begin{equation}\label{eq1.9}
\Upsilon_{\ell, r}(b_1, b_2; q_0) \lessapprox x^{-\e} Q^2N (q_0, \ell) q_0^{-2}
 \end{equation}
for $r \in [R, 2R]$, with $R$ satisfying the constraint \cite[(5.12)]{polymath8a}.

We get \eqref{eq1.9} by applying the first bound in \cite[Corollary 4.16]{polymath8a}. By simple considerations we can write
 \[e_r\left(\frac{ah}{nq_0q_1q_2}\right) e_{q_0q_1}
 \left(\frac{b_1h}{nrq_2}\right) = e_{rq_0q_1}
 \left(\frac{zh}{nq_2}\right)\]
with $(z, rq_0q_1) = 1$, $z = z(a, q_0, q_1, b_1, r, q_2)$.

We have 
 \begin{align}
&\sum_n C(n) \b(n) \overline{\b(n + \ell r)} \Phi_\ell (h, n, r, q_0, q_1, q_2)\label{eq1.10}\\
&\ll (q_0, \ell) \Bigg|\sum_{n\equiv t\ (\mod q_0)} \b(n) \overline{\b(n + \ell r)} e_{r q_0q_1} \left(\frac{zh}{nq_2}\right) e_{q_2} \left(\frac{b_2h}{(n+\ell r)r q_0q_1}\right)\Bigg|\notag
 \end{align}
for some $t\pmod {q_0}$ (see the remark about $C(n)$ preceding  \cite[(5.34)]{polymath8a}). We may therefore apply Corollary 4.16 with the following parameters:

\begin{enumerate} 
\item $\psi_N(n) = \b(n)\b(n+\ell r)$ (the required smoothness assumption is easy to verify in view of the smoothness of $\b(n)$). 

\item $d = q_0$, $d_1 = rq_0q_1$ and  $d_2 = q_2$.  These satisfy$(d_1, d_2) = 1$ since $(q_1,q_2)=1$ by assumption and $rq_0q_2$ is restricted to squarefree numbers in the definition of $\Upsilon$.  We therefore have $[d_1, d_2] = rq_0q_1q_2$.  In addition $[d_1,d_2]$ is $x^\delta$-smooth.

\item  $\d_i = d_i$ 

\item $\d_1' = \frac{d_1}{(q_0, d_1)} = rq_1$, $\d_2' = \frac{d_2}{(q_0, d_2)} = q_2$. 

\item  $c_1 = \frac{zh}{q_2}$, $(c_1, \d_1) = (h, q_1)$, $c_2 = \frac{b_2h}{rq_0q_1}$ and $(c_2, \d_2') = (h, q_2)$.
\end{enumerate}
Thus the left-hand side of \eqref{eq1.10} is
\begin{align*}
&\lessapprox  (q_0,\ell)\Bigl(q_0^{-1/2} N^{1/2} (rq_0q_1q_2)^{1/6} x^{\delta/6}+
q_0^{-1}\frac{(h,q_1q_2)}{rq_1q_2}N\Bigr)\\
&\lessapprox  (q_0,\ell)\Bigl(q_0^{-2/3} N^{1/2} (RQ^2)^{1/6} x^{\delta/6}+\frac{q_0(h,q_1q_2)N}{RQ^2}\Bigr)\\
\end{align*}
for $r\asymp R$, $q_i \asymp Q/q_0$. Summing over $h$, $q_1$, $q_2$, $1 \le |h| \le H \ll \frac{x^\e Q^2 R}{q_0M}$,
 \begin{align*}
\Upsilon_{\ell,r}(b_1, b_2;q_0) &\lessapprox&
\frac{x^{\e}(q_0,\ell) Q^4 R}{q_0^3M}\Bigl(q_0^{-2/3} N^{1/2} (RQ^2)^{1/6} x^{\delta/6}+\frac{q_0N}{RQ^2}\Bigr)\\
&\lessapprox& \frac{x^{\e}(q_0,\ell) Q^2}{q_0^2}\Bigl(q_0^{-5/3} N^{1/2} (RQ^2)^{7/6} x^{\delta/6}/M+N/M\Bigr).\\
\end{align*}
We now need 
$$\frac{N^{1/2} (RQ^2)^{7/6} x^{\delta/6}}{M}+\frac NM \lessapprox x^{-2\e} N.$$
We have $M\gtrapprox x^{1/2}$ so clearly $\frac{N}{M}\lessapprox Nx^{-2\e}$. Considering the first term, we must establish  
$$N^3 (RQ^2)^7 \lessapprox x^{6-\delta-12\e}.$$
We have 
$$RQ^2\ll R^{-1}x^{1+4\varpi}\lessapprox  x^{1+4\varpi+3\e+\delta}N^{-1}$$ 
so that 
$$N^3(RQ^2)^7\lessapprox N^{-4}x^{7+28\varpi+21\e+7\delta}.$$
We therefore require 
$$N^{-4}\lessapprox x^{-1-28\varpi-8\delta-31\e},$$
that is 
$$N\gtrapprox x^{1/4+7\varpi+2\delta+31\e/4}.$$
This follows from our assumptions \eqref{lem1cond1} and \eqref{lem1cond2} provided we choose a sufficiently small $\e$.  
 \end{proof}

We now derive the following result for ``Type II'' sums.  The reader should note that we use the definitions of Type I and Type II from Harman \cite{harmanbook} rather than those of Polymath \cite{polymath8a}.  

\begin{lem}\label{type2}
Suppose $0\leq \eta<\frac{2}{95}$.   Let $f$ be given by a convolution $f=\alpha*\beta$ where $\alpha$ and $\beta$ are coefficient sequences at scales $M$ and $N$.  Assume that $MN\asymp x$, 
$$x^{\frac25+\eta}\leq N\leq M\leq x^{\frac35-\eta}$$
and that $\a(m),\beta(n)$ satisfy the Siegel-Walfisz theorem. We conclude that $f$ has exponent of distribution $\theta_0(\eta)$ to smooth moduli.  
\end{lem} 

\begin{proof}
Taking $\sigma=\frac{1}{10}-\eta$ in part (iii) of \cite[Theorem 2.8]{polymath8a} we see that the condition 
$$\frac{160}{3} \varpi + 16\delta + \frac{34}{9} \sigma < 1$$
is satisfied provided that
$$\varpi <\frac{7}{600}-\frac{3\delta}{10}+\frac{17\eta}{240}$$
and 
$$64\varpi + 18 \delta + 2\sigma < 1$$
is satisfied if 
$$\varpi <\frac1{80}-\frac{9\delta}{32}+\frac{\eta}{32}.$$
By taking a sufficiently small $\delta$, the second of these is weaker than the first provided that $\eta<\frac{2}{95}$.  The condition 
$$68\varpi+14\delta<1$$
from part (iv) is satisfied provided that 
$$\varpi<\frac{1}{68}-\frac{7\delta}{34}.$$
For $\delta$ sufficiently small this is weaker than the above provided that $\eta<\frac{62}{1445}$. 

We conclude that any convolution satisfying the hypotheses of the Lemma may be handled satisfactorily using either part (iii) or (iv).  The result follows, recalling that $\theta=\frac12+2\varpi$.
\end{proof}

Combining this with Lemma \ref{lem1} we obtain a result for ``Type I'' sums.

\begin{lem}\label{type1}
Suppose $0\leq \eta<\frac{2}{95}$.   Let $f$ be given by a convolution $f=\alpha*\psi$ where $\alpha$ and $\psi$ are coefficient sequences at scales $M$ and $N$.  Assume that $MN\asymp x$, $\alpha$ satisfies the Siegel-Walfisz theorem, 
$$N\geq x^{199/600+119\eta/240}$$
and that $\psi(n)$ is smooth. We conclude that $f$ has exponent of distribution $\theta_0(\eta)$ to smooth moduli.
\end{lem}

\begin{proof}
We consider $3$ cases, depending on the size of $N$.

\begin{enumerate}
\item We can give a very trivial treatment for $N\geq x^{\frac35-\eta}$, see the discussion of Type 0 sums near the end of \cite[Section 3]{polymath8a}.  Specifically, if 
$$\frac{1}{2}+\frac{7}{300}+\frac{17\eta}{120}<\frac{3}{5}-\eta,$$
that is if $\eta<\frac{46}{685}$, then 
$$N\geq x^{\epsilon}q$$
for all $q\leq x^{1/2+7/300+17\eta/120-\epsilon}$ so the sum over $N$ may be evaluated asymptotically.

\item If $x^{2/5+\eta}\leq N\leq x^{3/5-\eta}$ then the result follows by Lemma \ref{type2}.

\item Finally, if $N\leq x^{2/5+\eta}$ we appeal to Lemma \ref{lem1}.  This requires that 
$$N\geq x^{1/4+7\varpi+2\delta},$$
which will be satisfied provided we choose a sufficiently small $\delta$ and 
$$N\geq x^{199/600+119\eta/240}.$$
\end{enumerate}
\end{proof}

Finally, using part (v) of \cite[Theorem 2.8]{polymath8a} we obtain the following estimate for a ``Type III'' sum.

\begin{lem}\label{type3}
Suppose $0\leq \eta<\frac{22}{3295}$.  Let $f=\a*\psi_1*\psi_2*\psi_3$ for coefficient sequences $\alpha,\psi_1,\psi_2,\psi_3$ at scales $M,N_1,N_2,N_3$.  Suppose that the $\psi_i$ are smooth and that the scales satisfy 
$$MN_1N_2N_3\asymp x,$$
$$N_1N_2,N_1N_3,N_2N_3\geq x^{3/5-\eta}$$
and
$$x^{1/5-2\eta}\leq N_1,N_2,N_3\leq x^{2/5+\eta}.$$
We conclude that $f$ has exponent of distribution $\theta_0(\eta)$ to smooth moduli.  
\end{lem}

\begin{proof}
We apply part (v) of \cite[Theorem 2.8]{polymath8a} with $\sigma=\frac{1}{10}-\eta$. 

In order for the lower bound 
$\s > \frac 1{18} + \frac{28}9\, \varpi+\frac{2\delta}{9}$ 
to hold we require a sufficiently small $\delta$ and 
$$\frac{1}{10}-\eta> \frac 1{18} + \frac{28}9\left(\frac{7}{600}+\frac{17\eta}{240}\right)$$
that is $\eta<\frac{22}{3295}$.  The condition $\varpi<\frac{1}{12}$ is clearly satisfied so the result follows.  
\end{proof}

When applying the results of this section to sums arising from the sieve it will be necessary to apply a suitable smooth partition of unity to decompose the summands into appropriate coefficient sequences.  The details of this procedure are given by Polymath in \cite[Section 3]{polymath8a} so we do not repeat them here.  It is also necessary to verify that the coefficient sequences satisfy the Siegel-Walfisz theorem but this can be established in each case by a suitable appeal to the theory of Dirichlet $L$-functions.

\section{The Sieve: Proof of Lemma \ref{minorant}}

For any natural number $n$ and any $y>0$ we define 
$$\psi(n,y)=\begin{cases}
1 & p|n\rightarrow p\geq y\\
0 & \text{otherwise}.\\
\end{cases}$$  
The indicator function of the primes in $[x,2x]$ is then given by\\ 
$\psi(n,(2x)^{1/2})$. 

If $y_1<y_2$ then the Buchstab identity is 
$$\psi(n,y_2)=\psi(n,y_1)-\twosum{n=p_1n_2}{y_1\leq p<y_2}\psi(n_2,p_1).$$ 
Letting $\l =(2x)^{1/5-2\eta}$ and writing $p_j=(2x)^{\a_j}$ this gives 
$$\psi(n,(2x)^{1/2})=\psi(n,\l)-\twosum{n=p_1n_2}{1/5-2\eta\leq \a_1< 1/2}\psi(n_2,p_1).$$
We assume for the remainder of this work that $\eta<\frac{22}{3295}$.  The full strength of this hypothesis will only be required in Lemma \ref{hblem} but there are numerous weaker requirements on $\eta$.  We will show, in Lemma \ref{fundlem}, that the function $\psi(n,\l)$ has exponent of distribution $\theta_0(\eta)$ to smooth moduli.  By Lemma \ref{type2} this also holds for the function 
$$-\twosum{n=p_1n_2}{2/5+\eta\leq \a_1< 1/2}\psi(n_2,p_1).$$
We write $\rho_0(n)$ for an arbitrary  function with exponent of distribution $\theta_0(\eta)$ to smooth moduli and we allow the meaning of $\rho_0(n)$ to differ at each occurrence.  With this convention we have 
$$\psi(n,(2x)^{1/2})=\rho_0(n)-\twosum{n=p_1n_2}{1/5-2\eta\leq \a_1< 2/5+\eta}\psi(n_2,p_1).$$
Applying Buchstab's identity again we obtain 
\begin{eqnarray*}
\psi(n,(2x)^{1/2})&=&\rho_0(n)-\twosum{n=p_1n_2}{1/5-2\eta\leq \a_1< 2/5+\eta}\psi(n_2,\l)\\
&&+\twosum{n=p_1p_2n_3}{1/5-2\eta\leq \a_2<\a_1< 2/5+\eta}\psi(n_3,p_2).\\
\end{eqnarray*}
The first unhandled term in this will be dealt with in Lemma \ref{fundlem}.  We can get a satisfactory treatment of that part of the second term with $\a_1+\a_2\in [2/5+\eta,3/5-\eta]$ by means of Lemma \ref{type2}.  In Lemma \ref{hblem} we will show that the part of the final term with $\a_1+\a_2\geq \frac35-\eta$ and $\a_2\geq \frac15+\frac{4\eta}{3}$ has exponent $\theta_0(\eta)$, for $\eta<\frac{22}{3295}$.  We therefore have 
$$\psi(n,(2x)^{1/2})=\rho_0(n)+\rho_1(n)+\rho_2(n),$$
where
$$\rho_1(n)=\threesum{n=p_1p_2n_3}{1/5-2\eta\leq \a_2<\a_1< 2/5+\eta}{\a_1+\a_2<2/5+\eta}\psi(n_3,p_2)$$
and
$$\rho_2(n)=\threesum{n=p_1p_2n_3}{1/5-2\eta\leq \a_2<\a_1< 2/5+\eta}{\a_1+\a_2>3/5-\eta,\a_2<1/5+4\eta/3}\psi(n_3,p_2).$$
We apply Buchstab's identity twice to $\rho_1$ to obtain 
\begin{eqnarray*}
\rho_1(n)&=&\threesum{n=p_1p_2n_3}{1/5-2\eta\leq \a_2<\a_1< 2/5+\eta}{\a_1+\a_2<2/5+\eta}\psi(n_3,\l)\\
&&-\threesum{n=p_1p_2p_3n_4}{1/5-2\eta\leq \a_3<\a_2<\a_1< 2/5+\eta}{\a_1+\a_2<2/5+\eta}\psi(n_4,\l)\\
&&+\threesum{n=p_1p_2p_3p_4n_5}{1/5-2\eta\leq \a_4<\a_3<\a_2<\a_1< 2/5+\eta}{\a_1+\a_2<2/5+\eta}\psi(n_5,p_4).
\end{eqnarray*}
By Lemma \ref{fundlem} the first term in this has exponent $\theta_0(\eta)$.  This also applies to the second term provided that $\a_3<\zeta$ ($\zeta$ will be defined in Lemma \ref{fundlem}).  We know that 
$$\a_3\leq \frac{\a_1+\a_2}{2}<\frac15+\frac{\eta}{2}$$
and therefore $\a_3<\zeta$ provided that 
$$\frac{1}{5}+\frac{\eta}{2}<\frac{161}{600}-\frac{359\eta}{240},$$ 
that is $\eta<\frac{82}{2395}$.  We conclude that 
$$\rho_1(n)=\rho_0(n)+\threesum{n=p_1p_2p_3p_4n_5}{1/5-2\eta\leq \a_4<\a_3<\a_2<\a_1< 2/5+\eta}{\a_1+\a_2<2/5+\eta}\psi(n_5,p_4)=\rho_0(n)+\rho_3(n),$$ 
say.  We will deal with $\rho_3(n)$ in Lemma \ref{rho3}.

We can only apply Buchstab's identity once to the function $\rho_2(n)$. We obtain 
\begin{eqnarray*}
\rho_2(n)&=&\threesum{n=p_1p_2n_3}{1/5-2\eta\leq \a_2<\a_1< 2/5+\eta}{\a_1+\a_2>3/5-\eta,\a_2<1/5+4\eta/3}\psi(n_3,\l)\\
&&-\threesum{n=p_1p_2p_3n_4}{1/5-2\eta\leq \a_3<\a_2<\a_1< 2/5+\eta}{\a_1+\a_2>3/5-\eta,\a_2<1/5+4\eta/3}\psi(n_4,p_3).
\end{eqnarray*}
We may apply Lemma \ref{fundlem} to the first term of this provided that $\a_2<\zeta$, which is satisfied if 
$$\frac{1}{5}+\frac{4\eta}{3}<\frac{161}{600}-\frac{359\eta}{240}\longleftrightarrow \eta< \frac{82}{3395}.$$
We therefore have
$$\rho_2(n)=\rho_0(n)-\threesum{n=p_1p_2p_3n_4}{1/5-2\eta\leq \a_3<\a_2<\a_1< 2/5+\eta}{\a_1+\a_2>3/5-\eta,\a_2<1/5+4\eta/3}\psi(n_4,p_3)=\rho_0(n)+\rho_4(n),$$
say.  The function $\rho_4(n)$ will be treated in Lemma \ref{rho4}.  

We now record the various results coming from the Harman sieve which are needed in the above.

\begin{lem}\label{fundlem}
For $0\leq \eta<\frac{1}{100}$ let 
$$\zeta=\frac{161}{600}-\frac{359\eta}{240}.$$
We conclude that all of the following functions have exponent of distribution $\theta_0(\eta)$ to smooth moduli:
\begin{enumerate}
\item $\psi(n,\l)$

\item 
$$\twosum{n=p_1n_2}{1/5-2\eta\leq \a_1< 2/5+\eta}\psi(n_2,\l)$$

\item 
$$\threesum{n=p_1p_2n_3}{1/5-2\eta\leq \a_2<\a_1<2/5+\eta}{\a_1+\a_2<2/5+\eta}\psi(n_3,\l)$$
  
\item 
$$\threesum{n=p_1p_2n_3}{1/5-2\eta\leq \a_2<\a_1<2/5+\eta}{\a_1+\a_2>3/5-\eta,\a_2<\zeta}\psi(n_3,\l)$$

\item 
$$\threesum{n=p_1p_2p_3n_4}{1/5-2\eta<\a_3< \a_2 < \a_1<2/5+\eta}{\a_1+\a_2<\frac25+\eta,\a_3<\zeta}\psi(n_4,\l).$$

\item 
$$\threesum{n=n_1p_2p_3p_4}{1/5-2\eta\leq \a_3<\a_2< 2/5+\eta}{\a_3+\a_4<2/5+\eta,\a_2<\zeta,\a_4>\a_3}\psi(n_1,\l)$$
\end{enumerate}
\end{lem}

\begin{proof}
We use the Harman sieve in the form given by Baker and Weingartner \cite[Lemma 14]{bakerw}.  In the notation of that lemma we let 
$\a=\frac25+\eta$
and $\b=\frac15-2\eta$, so that $\a+\b=\frac35-\eta$. We take 
$$M=x^{401/600-119\eta/240}.$$ 
We must therefore take 
$$R<x^{2/5+\eta}\text{ and }S<MX^{-\a}=x^{161/600-359\eta/240}.$$
We write $S=x^\zeta$ so that 
$$\zeta=\frac{161}{600}-\frac{359\eta}{240},$$
as in the statement of the lemma.  

We apply the lemma to functions of the form
 $$w(n)=\mathbf{1}_{n\equiv a\pmod q}-\frac{1}{\phi(q)}\mathbf{1}_{(n,q)=1}.$$
Lemma \ref{type1} shows that the condition \cite[(4.1)]{bakerw} is satisfied for almost all relevant $q$.  Specifically we can take $y=\frac{x}{q}\log^{-A}x$, for any $A>0$, and the cardinality of the exceptional set of $q$ is $O_{A,B}(x\log^{-B}x)$, for any $B>0$. By using Lemma \ref{type2} we see that an identical conclusion holds for the condition \cite[(4.2)]{bakerw}.  

Observe that $\l=x^\b$ and therefore, in order to complete the proof, it only remains to show that suitable choices of $R,S,u_r,v_s$ can be found in each instance.

\begin{enumerate}
\item We take $R=S=1$ and $u_1=v_1=1$.

\item We take $R=x^{2/5+\eta}$, $S=1$, $v_1=1$ and $u_r$ the indicator function of the primes in $[x^{1/5-2\eta},x^{2/5+\eta}]$. 

\item Since $\a_1+\a_2<\frac25+\eta$ and $\a_2<\a_1$ we have $\a_2<\frac15+\frac{\eta}{2}$.  This is smaller than $\zeta$ provided that $\eta<\frac{82}{2395}$. The result therefore follows if we take $R=x^{2/5+\eta},S=x^{1/5+\eta/2}$ and $u_r,v_s$ suitable indicator functions of primes.  A finer than dyadic decomposition can be used to remove the constraint on $\a_1+\a_2$.  

\item This can be handled in an analogous way to the previous part since the constraint $\a_2<\zeta$ is assumed to hold.

\item We now take $R=x^{2/5+\eta},S=x^\zeta$, $u_r$ is the indicator function of certain products of two primes and $v_s$ the indicator of certain primes. 

\item This is very similar to the previous part except for a different labelling of the variables.  We take $R=x^{2/5+\eta}$, $S=x^\zeta$, $u_r$ the indicator of the products $p_3p_4$ and $v_s$ the indicator of the primes $p_2$.  
\end{enumerate}

The result therefore follows if $\eta$ satisfies the above constraint $\eta<\frac{82}{2395}$ as well as the requirements of Lemmas \ref{type1} and \ref{type2}.  Certainly any $\eta<\frac{1}{100}$ will be satisfactory.  
\end{proof}  

In order to prove Lemma \ref{hblem}, below, we need some purely combinatorial results.  

 \begin{lem}\label{lem2}
Let $0 < \eta < \frac{82}{5395}$. Let $\g_1 \ge \cdots \ge \g_t > 0$, $\g_1 + \cdots + \g_t = 1$ and
 \begin{enumerate}
\item[(a)] $\g_1 <\frac{199}{600}+\frac{119\eta}{240}$;

\item[(b)] no sum $\sum\limits_{j\in \mc S} \g_j$ is in $\left[\frac 25 + \eta, \, \frac 35 - \eta\right]$.

\item[(c)] {\rm either} $\g_3 < \frac 15 - 2\eta$ {\rm or} $\g_2 + \g_3 < \frac 25 + \eta$. 
 \end{enumerate}
Then $\g_5 \ge \frac 15 - 2\eta$ and
 $$\g_1 + \g_2 + \g_6 + \cdots + \g_t <
 \frac 25 + \eta.$$
 \end{lem}
 
 \begin{proof}
We first show that 
 \begin{equation}\label{eq1.12}
\g_1 + \g_2 < \frac 25 + \eta. 
 \end{equation}
Suppose the contrary, then by (b) we must have $\g_1 + \g_2 > \frac 35 - \eta$. It is clear that $\g_3 \ge \frac 15 - 2\eta$, since otherwise
 \[\g_1 + \g_3 + \cdots + \g_m \quad
 (3 \le m \le t)\]
are successively found to be in $\left[0, \frac 25 + \eta\right]$ giving $\g_2 \ge \frac 35 - \eta$ which contradicts (a).  We may therefore deduce from (c) that $\g_2 + \g_3 < \frac 25 + \eta$.

Now 
$$\g_1 - \g_3 = (\g_1+\g_2)-(\g_2+\g_3)>(\frac35-\eta)-(\frac25+\eta)=\frac15-2\eta$$
so that 
$$\g_1> \frac15-2\eta+\g_3\geq \frac15-2\eta+\frac15-2\eta=\frac25-4\eta.$$
Since $\eta<\frac{82}{5395}$ this contradicts the assumption (a) that 
$$\g_1<\frac{199}{600}+\frac{119\eta}{240}$$
and therefore proves \eqref{eq1.12} by contradiction.

We now evaluate the largest $m$ for which $\g_m \ge \frac15 - 2\eta$.  If $m \le 4$, then $\g_1 + \g_2 + \g_{m+1} + \cdots + \g_t < \frac25 + \eta$ (arguing as above). Since $\g_3 + \g_4\leq \g_1+\g_2 < \frac25 + \eta$, we get $\g_1 + \cdots + \g_t < 1$ if $m \le 4$; so $m \ge 5$.

If $m \ge 6$, then, since $\eta<\frac{1}{60}$,  
$$\g_1+\ldots+\g_6\geq \frac65-12\eta>1.$$
This contradiction shows that $m$ must be $5$. We now conclude as above that
 \[\g_1 + \g_2 + \g_6 + \cdots +
 \g_t < \frac25 + \eta. \qedhere\]
 \end{proof}
 
 \begin{lem}\label{lem3}
   Let $\frac 15 - 2\eta \le \a_2 < \a_1 < \frac 25 + \eta$ and $\a_2 \le \frac 13$. Suppose that $0 < \eta < \frac{82}{5395}$. Suppose that
 \begin{align*}
\a_1 &= \b_1 + \cdots + \b_r \ , \ \b_1 \ge \cdots \ge \b_r > 0\\
\a_2 &= \b_{r+1} + \cdots + \b_s \ , \ \b_{r+1} \ge \cdots \ge \b_s > 0,\\
1 - \a_1-\a_2 &= \b_{s+1} + \cdots + \b_t \ , \ \b_{s+1} \ge \cdots \ge \b_t > 0. 
 \end{align*}
Let $\g_1 \ge \cdots \ge \g_t$ be the reordering of $\b_1, \ldots,\b_t$ in decreasing order. Suppose that
 \begin{enumerate}
\item[(a)] $\g_1 <\frac{199}{600}+\frac{119\eta}{240}$;

\item[(b)] no sum $\sum\limits_{j\in \mc S} \g_j$ is in $\left[\frac 25 + \eta, \, \frac 35 - \eta\right]$,

\item[(c)] either $\g_3 < \frac 15 - 2\eta$ or $\g_2 + \g_3 < \frac 25 + \eta$.
 \end{enumerate}
Then {\rm either}
 \[\a_1 + \a_2 < \frac 25 + \eta\]
{\rm or}
 \[\a_1 + \a_2 > \frac 35 - \eta, \ , \
 \a_2 < \frac 15 + \frac{4\eta}3.\]
 \end{lem}
 
 \begin{proof}
We may suppose that $\a_1 + \a_2 \ge \frac 25 + \eta$, hence $\a_1 + \a_2 > \frac 35 - \eta$ by (b). By \Cref{lem2}, we have $\g_5 \ge \frac 15 - 2\eta$ and
 \begin{equation}\label{eq1.13}
\g_1 + \g_2 + \g_6 + \cdots + \g_t< \frac 2{5} + \eta. 
 \end{equation}
We deduce from this that 
\begin{equation}\label{eq:comp}
\g_3+\g_4+\g_5>\frac35-\eta.
\end{equation}
Since $\a_2 \leq \frac 13$ and $\g_3\leq \g_2 < \frac 15 + \frac \eta 2$ from \eqref{eq1.13} we obtain
 \[\g_4+\g_5=\g_3+\g_4+\g_5-\g_3 > \frac 35 - \eta - \left(
 \frac 15 + \frac \eta 2\right) = \frac 25
 - \frac {3\eta}2 > \frac 13 \geq \a_2.\]
Therefore \textit{at most} one of $\g_1, \ldots, \g_5$ can be found in $\{\b_{r+1}, \ldots, \b_s\}$. Hence, by \eqref{eq1.13}, 
 \[\a_2 \le \g_1 + \g_6 + \ldots+\g_t <
 \frac 2{5} + \eta - \g_2.\]
Using \eqref{eq:comp} we have 
 \[\g_2 \ge \frac 13\, (\g_3 + \g_4 + \g_5)
 > \frac 15 - \frac\eta 3\]
so $\a_2 < \frac 15 + \frac{4\eta}3$.
 \end{proof}

The following lemma is very significant in our argument as it makes crucial use of the Type III information given in Lemma \ref{type3}.  Without this result we would need a much wider Type II interval.

\begin{lem}\label{hblem}
Suppose $\eta<\frac{22}{3295}$.   The function 
$$\threesum{n=p_1p_2n_3}{1/5-2\eta\leq \a_2<\a_1< 2/5+\eta}{\a_1+\a_2> 3/5-\eta,\a_2\geq 1/5+4\eta/3}\psi(n_3,p_2)$$
has exponent of distribution $\theta_0(\eta)$ to smooth moduli.  
\end{lem}

\begin{proof}
We begin by observing that $\psi(n_3,p_2)=0$ if $p_2>n_3$.  We may therefore impose the further constraint 
$\a_1+2\a_2\leq 1$.  This implies that $\a_2\leq \frac13$. Next we show that the only summands which contribute to the function have $n_3$ prime.  This follows if $\a_1+3\a_2>1$ which may be verified since 
$$\a_1+3\a_2=\a_1+\a_2+2\a_2> \frac35-\eta+2(\frac{1}{5}+\frac{4\eta}{3})=1+\frac{5\eta}{3}.$$
It follows that the function under consideration is 
$$\threesum{n=p_1p_2p_3}{1/5-2\eta\leq \a_2<\a_1< 2/5+\eta,\a_2\leq \a_3}{\a_1+\a_2> 3/5-\eta,1/5+4\eta/3\leq \a_2\leq 1/3}1.$$

We now apply the Heath-Brown identity \cite{rhbidentity} to each prime $p_i$.  After a relatively simple treatment of the proper prime powers and a finer than dyadic subdivision to remove the weights $\log p_i$ we may decompose our function as a sum of $\log^{O(1)}x$ summands  (see Polymath \cite[Section 3]{polymath8a} for some similar arguments).  We may then use the previous lemma to show that all of these summands may be handled by one of Lemmas \ref{type2}, \ref{type1} or \ref{type3}.  
\end{proof}

We now turn our attention to the remaining functions $\rho_3(n)$ and $\rho_4(n)$.

\begin{lem}\label{rho3}
We have 
$$\rho_3(n)=\rho_0(n)+\rho_5(n)$$
for a positive function $\rho_5(n)$ which satisfies 
$$\sum_{x\leq n\leq 2x}\rho_5(n)=\frac{x}{\log x}\left(\int_{E(\eta)} f(\bs \a)d\bs \a+o(1)\right),$$
where the integrand and region of integration are as defined in Lemma \ref{minorant}.
\end{lem}

\begin{proof}
Recall that 
$$\rho_3(n)=\threesum{n=p_1p_2p_3p_4n_5}{1/5-2\eta\leq \a_4<\a_3<\a_2<\a_1< 2/5+\eta}{\a_1+\a_2<2/5+\eta}\psi(n_5,p_4).$$
If $p_4>n_5$ then $\psi(n_5,p_4)=0$ so we may impose the additional constraint $\a_1+\a_2+\a_3+2\a_4\leq 1$.  

Any part of $\rho_3$ for which a sum of some of the $\a_i$ lies in the Type II range $[\frac25+\eta,\frac35-\eta]$ may be handled by Lemma \ref{type2}.  Since $\a_1+\a_2<\frac25+\eta$ we know that no sum of $1$ or $2$ variables can lie in that range.  It can be seen that a sum of $3$ variables lies in the Type II range if and only if $\a_2+\a_3+\a_4\leq \frac35-\eta$.  We therefore obtain 
$$\rho_3(n)=\rho_0(n)+\rho_5(n)$$
with 
$$\rho_5(n)=\threesum{n=p_1p_2p_3p_4n_5}{1/5-2\eta\leq \a_4<\a_3<\a_2<\a_1< 2/5+\eta}{\a_1+\a_2<2/5+\eta,\a_1+\a_2+\a_3+2\a_4\leq 1,\a_2+\a_3+\a_4>3/5-\eta}\psi(n_5,p_4).$$
We now claim that the only nonzero summands in $\rho_5$ come from prime values of $n_5$.  This will follow if $\a_1+\a_2+\a_3+3\a_4>1$ which holds since $\eta<\frac{1}{60}$ and  
$$\a_1+\a_2+\a_3+3\a_4>6\a_4\geq \frac65-12\eta.$$
We may now estimate 
$$\sum_{x\leq n\leq 2x}\rho_5(n)$$
using standard techniques and the result follows.  
\end{proof}

Before dealing with $\rho_4$ we need a further combinatorial fact.

\begin{lem}\label{permlem}
Suppose $\a_1>\ldots>\a_5$ are real numbers.

\begin{enumerate}
\item There are precisely $4$ permutations $\b_1,\ldots,\b_5$ of the $\a_i$ for which 
$$\b_1>\b_2>\b_3>\b_4\text{ and }\b_4<\b_5.$$

\item There are precisely $20$ permutations $\b_1,\ldots,\b_5$ with 
$$\b_1>\b_2,\ \b_2<\b_3,\text{ and }\b_4<\b_5.$$
\end{enumerate}
\end{lem}

\begin{proof}
  \begin{enumerate}
  \item We must choose $\b_5$ from $\{\a_1,\ldots,\a_4\}$ and for each such choice there is precisely $1$ permissible permutation of the remaining variables.

\item There are precisely $10$ choices for $\b_4,\b_5$, namely any pair $\a_i,\a_j$ with $i>j$.  There are then exactly two orderings of the remaining $\a$ to give $\b_1,\b_2,\b_3$.  
  \end{enumerate}
\end{proof}

\begin{lem}\label{rho4}
We have 
$$\rho_4(n)=\rho_0(n)+\rho_6(n)$$ 
where $\rho_6(n)$ is positive and 
$$\sum_{x\leq n\leq 2x}\rho_6(n)=\frac{x}{\log x}\left(5\int_{E(\eta)} f(\bs \a)d\bs \a+o(1)\right).$$
\end{lem}

\begin{proof}
Recall that 
$$\rho_4(n)=-\threesum{n=p_1p_2p_3n_4}{1/5-2\eta\leq \a_3<\a_2<\a_1< 2/5+\eta}{\a_1+\a_2>3/5-\eta,\a_2<1/5+4\eta/3}\psi(n_4,p_3).$$
We observe that 
$$\a_1+\a_2+3\a_3>\frac35-\eta+3\a_3>\frac65-7\eta$$
which is greater than $1$ if $\eta<\frac{1}{35}$.  It follows that the only nonzero terms in $\rho_4$ have $n_4$ prime and $\a_4>\a_3$:
$$\rho_4(n)=-\threesum{n=p_1p_2p_3p_4}{1/5-2\eta\leq \a_3<\a_2<\a_1< 2/5+\eta}{\a_1+\a_2>3/5-\eta,\a_2<1/5+4\eta/3,\a_4>\a_3}1.$$
The condition $\a_1>\a_2$ may be dropped since we have 
$$\a_1=(\a_1+\a_2)-\a_2>\frac35-\eta-(\frac15+\frac{4\eta}{3})=\frac25-\frac{7\eta}{3}.$$
In addition we have 
$$\a_1\leq 1-3(\frac15-2\eta)=\frac25+6\eta.$$
Therefore, if we use Lemma \ref{type2} on the range $\a_1\in [\frac25+\eta,\frac25+6\eta]$ we may drop the constraint $\a_1\leq \frac25+\eta$ to obtain 
$$\rho_4(n)=\rho_0(n)-\threesum{n=p_1p_2p_3p_4}{1/5-2\eta\leq \a_3<\a_2< 2/5+\eta}{\a_3+\a_4<2/5+\eta,\a_2<1/5+4\eta/3,\a_4>\a_3}1$$
(in which the condition $\a_1+\a_2>\frac35-\eta$ has been replaced with the equivalent $\a_3+\a_4<\frac25+\eta$).

We now perform a reversal of roles, replacing the prime $p_1$ by an integer $n_1$ whose primality is detected with the sieve.  Writing $n_1=(2x)^{\a_1}$ this yields 
$$\rho_4(n)=\rho_0(n)+\rho'_4(n)$$
with 
$$\rho'_4(n)=-\threesum{n=n_1p_2p_3p_4}{1/5-2\eta\leq \a_3<\a_2< 2/5+\eta}{\a_3+\a_4<2/5+\eta,\a_2<1/5+4\eta/3,\a_4>\a_3}\psi(n_1,n_1^{1/2}).$$
This will be compared with 
$$\rho''_4(n)=-\threesum{n=n_1p_2p_3p_4}{1/5-2\eta\leq \a_3<\a_2< 2/5+\eta}{\a_3+\a_4<2/5+\eta,\a_2<1/5+4\eta/3,\a_4>\a_3}\psi(n_1,\l).$$
By the final part of Lemma \ref{fundlem} the function $\rho''_4(n)$ has exponent of distribution $\theta_0(\eta)$ to smooth moduli.  We observe that 
$$\a_2+\a_3+\a_4+3(\frac15-2\eta)\geq 6(\frac15-2\eta)=\frac65-12\eta>1,$$
since $\eta<\frac{1}{60}$.  This means that for any nonzero term in $\rho''_4(n)$ we have $n_1<\l^3$.  Therefore, if a term occurs in $\rho''_4(n)$ but not in $\rho'_4(n)$ it must have $n_1=p_5p_6$ with $\l\leq p_5\leq p_6$.  We therefore conclude that 
$$\rho_4(n)=\rho_0(n)+\rho'_6(n)$$
with 
$$\rho'_6(n)=\threesum{n=p_2p_3p_4p_5p_6}{1/5-2\eta\leq \a_3<\a_2< 2/5+\eta}{\a_3+\a_4<2/5+\eta,\a_2<1/5+4\eta/3,\a_4>\a_3,1/5-2\eta\leq\a_5\leq \a_6}1.$$
We let $\rho_6(n)$ be the same function as $\rho'_6(n)$ but with the added constraint that no sum of the variables $\a_2,\ldots,\a_6$ is in $[\frac25+\eta,\frac35-\eta]$.  By Lemma \ref{type2} we have 
$$\rho'_6(n)=\rho_0(n)+\rho_6(n).$$
Next we show that in $\rho_6(n)$ the condition $\a_2<\frac15+\frac{4\eta}{3}$ is redundant.  We begin by observing that, in $\rho'_6(n)$, any sum of $3$ of the variables exceeds $\frac35-6\eta$ and therefore, in $\rho_6(n)$, any sum of $2$ variables must be smaller than $\frac25+\eta$. We therefore obtain 
$$3\a_2+\a_3+\a_4+\a_5=(\a_2+\a_3)+(\a_2+\a_4)+(\a_2+\a_5)<\frac65+3\eta.$$ 
In addition, any sum of $3$ variables in $\rho_6(n)$ must exceed $\frac35-\eta$ and therefore 
$$3\a_2<\frac65+3\eta-(\frac35-\eta)=\frac35+4\eta\longrightarrow \a_2<\frac15+\frac{4\eta}{3}.$$
We conclude that 
$$\rho_4(n)=\rho_0(n)+\rho_6(n)$$
with 
$$\rho_6(n)=\threesum{n=p_2p_3p_4p_5p_6}{1/5-2\eta\leq \a_3<\a_2,1/5-2\eta\leq\a_5\leq \a_6,\a_4>\a_3}{\a_i+\a_j<2/5+\eta\ \forall i<j,\a_i+\a_j+\a_k>3/5-\eta\ \forall i<j<k}1.$$
To complete the proof we will show that 
$$\sum_{x\leq n\leq 2x}\rho_6(n)=(5+o(1))\sum_{x\leq n\leq 2x}\rho_5(n)$$
with $\rho_5(n)$ as in Lemma \ref{rho3}:
$$\rho_5(n)=\threesum{n=p_1p_2p_3p_4p_5}{1/5-2\eta\leq \a_4<\a_3<\a_2<\a_1,\a_5\geq \a_4}{\a_i+\a_j<2/5+\eta\ \forall i<j,\a_i+\a_j+\a_k>3/5-\eta\ \forall i<j<k}1.$$
We observe that 
\begin{eqnarray*}
\#\{n\in [x,2x]:p^2|n\text{ for some }p\in [\l,x^{2/5+\eta}]\}&\leq& \sum_{\l\leq n\leq x^{2/5+\eta}}\frac{x}{n^2}\\
&\ll& \frac{x}{\l}=x^{4/5+2\eta}.\\
\end{eqnarray*}
It follows, since $\rho_5(n),\rho_6(n)\ll 1$, that the contribution of $n$ which are not squarefree to both sums is $O(x^{4/5+2\eta})$ and hence negligible.  

When restricted to squarefree numbers, the functions $\rho_5(n)$ and $\rho_6(n)$ have the same support:
\begin{equation*}
\begin{split}
\{n:n=p_1p_2p_3p_4p_5,\frac15-2\eta\leq \a_1<\ldots<\a_5<\frac25+\eta,\\
\a_i+\a_j<\frac25+\eta\ \forall i<j,\a_i+\a_j+\a_k>\frac35-\eta\ \forall i<j<k\}.
\end{split}
\end{equation*}
For $n$ in this set, Lemma \ref{permlem} shows that $\rho_5(n)=4$ and $\rho_6(n)=20$.  The result follows.  
\end{proof}

We may now complete the proof of Lemma \ref{minorant}.  We have shown that 
$$\rho(n,(2x)^{1/2})=\rho_0(n)+\rho_5(n)+\rho_6(n),$$
$\rho_5(n)+\rho_6(n)\geq 0$ and 
$$\sum_{x\leq n\leq 2x}(\rho_5(n)+\rho_6(n))=\frac{x}{\log x}\left(6\int_{E(\eta)} f(\bs \a)d\bs \a+o(1)\right).$$
The lemma follows on taking 
$$\rho(n)=\rho(n,(2x)^{1/2})-\rho_5(n)-\rho_6(n),$$
which is supported on integers all of whose prime factors exceed $\l$.  

\section{Proof of Lemma \ref{sievelem}}

We will describe the necessary changes to the arguments of Polymath \cite{polymath8b}.  A similar argument is also given in the work of Baker and Zhao \cite{bakerz}.  Define $\mathrm{DHL}[k,m]$ as in \cite[Claim 3.1]{polymath8b}.  In order to show that 
 $$H_m\ll m\exp(c_0m),$$
for a given constant $c_0$, it suffices to show that $\mathrm{DHL}[k,m+1]$ holds whenever $m\geq 1$ and $k\geq C\exp(c_0m)$ (for some sufficiently large absolute constant $C$).  

The claim $\mathrm{DHL}[k,m+1]$ will be established using \cite[Lemma 3.4]{polymath8b}.  The difference in our approach is our method for obtaining a lower bound for the sum in \cite[(14)]{polymath8b}.  Specifically, if $\rho(n)$ is a minorant for the indicator function of the primes we have 
$$\twosum{x\leq n\leq 2x}{n\equiv b\pmod W}\nu(n)\theta(n+h_i)\geq \log x\twosum{x\leq n\leq 2x}{n\equiv b\pmod W}\nu(n)\rho(n+h_i),$$
for any nonnegative function $\nu(n)$.  We will give an asymptotic formula for the sum on the right hand side.

Our sieve function $\nu(n)$ will be of an identical form to those used in \cite{polymath8b}, namely the square of a linear combination of the divisor sums \cite[(16)]{polymath8b}. We therefore require asymptotics for sums of the form 
$$\twosum{x\leq n\leq 2x}{n\equiv b\pmod W}\rho(n+h_i)\prod_{i'=1}^k\lambda_{F_{i'}}(n+h_{i'})\lambda_{G_{i'}}(n+h_{i'}).$$ 
We suppose, as in the statement of Lemma \ref{sievelem}, that $\rho(n)$ is supported on numbers all of whose prime factors exceed $x^\xi$. Then, if the functions $F_i,G_i$ are supported on $[0,\xi]$, it is enough to evaluate sums of the form  
$$\twosum{x\leq n\leq 2x}{n\equiv b\pmod W}\rho(n+h_i)\prod_{1\leq i'\leq k,i'\ne i}\lambda_{F_{i'}}(n+h_{i'})\lambda_{G_{i'}}(n+h_{i'})$$ 
(compare with \cite[(21)]{polymath8b}).  In our case, this sum will be handled by a modification of \cite[Theorem 3.5]{polymath8b}.  For a function $F$ we write, as in \cite{polymath8b},
$$S(F)=\sup\{x:f(x)\ne 0\}.$$

\begin{lem}
Let $k \geq 2$ be fixed, let $(h_1,\dots,h_k)$ be a fixed admissible $k$-tuple, and let $b\pmod W$ be such that $b+h_i$ is coprime to $W$ for each $i=1,\dots,k$.  Let $1 \leq i_0 \leq k$ be fixed, and for each $1 \leq i \leq k$ distinct from $i_0$, let $F_{i}, G_{i}: [0,+\infty) \to \R$ be fixed smooth compactly supported functions.

Let $\rho(n):[x,2x]\rightarrow \R$ be a function which is equidistributed in arithmetic progressions to squarefree, $x^\delta$-smooth moduli $q\leq x^\theta$.  Suppose that $\rho(n)$ is supported on integers having no prime factors smaller than $x^\d$ and that 
$$\sum_{n\leq x\leq 2x}\rho(n)=\frac{x}{\log x}(1-c_1+o(1)),$$
for a constant $c_1$.

Finally, suppose that if $i\ne i_0$ we have 
$$S(F_i),S(G_i)<\delta$$
and that 
$$\twosum{1\leq i\leq k}{i\ne i_0}(S(F_i)+S(G_i))<\theta.$$
We may then conclude that 
\begin{equation*}
\begin{split}
\twosum{x\leq n\leq 2x}{n\equiv b\pmod W}\rho(n+h_{i_0})\prod_{1\leq i\leq k,i\ne i_0}\lambda_{F_i}(n+h_i)\lambda_{G_i}(n+h_i)\\
=(c(1-c_1)+o(1))B^{1-k}\frac{x}{\phi(W)\log x},
\end{split}
\end{equation*}
with $B$ as in \cite[(12)]{polymath8b} and $c$ as in \cite[Theorem 3.5]{polymath8b}.  
\end{lem}

\begin{proof}
This is almost identical to the work in \cite[Sections 4.3 and 4.4]{polymath8b}.  The only difference is that rather than an appeal to the Prime Number Theorem we use our assumptions on $\rho$ to obtain 
$$\twosum{x\leq n\leq 2x}{(n,q)=1}\rho(n)=\sum_{x\leq n\leq 2x}\rho(n)=\frac{x}{\log x}(1-c_1+o(1)),$$
for any $x^\delta$-smooth $q$.  
\end{proof}

Using this result and the arguments of \cite[Section 5.2]{polymath8b} we may obtain the following modification of \cite[3.10]{polymath8b}.

\begin{lem}
Let $k\geq 2$ and $m\geq 1$ be fixed integers.  Suppose there exists a function $\rho(n)$ satisfying all the hypotheses of the previous lemma.  If 
$$M_k^{[\frac{2\delta}{\theta}]} > \frac{2m}{\theta(1-c_1)},$$
with $M_k^{[\alpha]}$ as in \cite[Theorem 3.10]{polymath8b}, then $\mathrm{DHL}[k,m+1]$ holds.
\end{lem}

It is shown in \cite[Section 6]{polymath8b} that for any $\alpha>0$ we have 
$$M_k^{[\alpha]}\geq \log k-O_\alpha(1).$$
Therefore, given a $\rho(n)$ as in the previous two lemmas, we may establish $\mathrm{DHL}[k,m+1]$ provided that 
$$\log k > \frac{2m}{\theta(1-c_1)}+O_{\theta,\delta}(1),$$
that is 
$$k\geq C\exp(c_0m)$$
with 
$$c_0=\frac{2}{\theta(1-c_1)}$$
and a constant $C$ which may depend on $\theta$ and $\delta$.  

Finally we suppose that $\rho(n)$ satisfies the hypotheses of Lemma \ref{sievelem}, so that it has exponent of distribution $\theta$ to smooth moduli.  Then, for any $\epsilon>0$ there exists a $\delta>0$  such that $\rho(n)$ is equidistributed in arithmetic progressions to $x^\d$-smooth moduli $q\leq x^{\theta-\epsilon}$.  If necessary we may replace $\delta$ by $\min(\delta,\xi)$ so that $\rho(n)$ is supported on numbers with no prime factor smaller than $x^\delta$.  The above then shows that $\mathrm{DHL}[k,m+1]$ holds for 
$$k\geq C\exp(c_0m)$$
with 
$$c_0=\frac{2}{(\theta-\epsilon)(1-c_1)}$$
and a $C$ depending on $\epsilon$.  By taking a sufficiently small $\epsilon$ we conclude that for any
$$c_0>\frac{2}{\theta(1-c_1)}$$
we have 
$$H_m\ll m\exp(c_0m).$$
Lemma \ref{sievelem} follows since the factor $m$ may be removed by working with a slightly smaller $c_0$.

\section{Proof of Theorem \ref{mainthm}}

We compute an upper bound for 
$$c_1(\eta)=6\int_{E(\eta)} f(\bs \a)d\bs \a$$
for $\eta=\frac{22}{3295}$.  A computer calculation shows that the volume of $E(\frac{22}{3295}$ is at most $3\times 10^{-10}$.  In addition, the maximum of the integrand in that region is 
$$\left(\frac15-2\eta\right)^{-5}\leq 4415.$$
It follows that 
$$c_1(\frac{22}{3295})<6\times 3\times 10^{-10}\times 4415<8\times 10^{-6}$$ 
and thus 
$$\theta(\eta)(1-c_1(\eta))>\left(\frac12+\frac{7}{300}+\frac{17}{120}\times \frac{22}{3295}\right)\left(1-8\times 10^{-6}\right)>0.52427.$$ 
Theorem \ref{mainthm} now follows if we choose a $\eta$ sufficiently close to $\frac{22}{3295}$, since $\frac{2}{0.52427}<3.815$.

We note that the above integral could be computed much more accurately.  However, since our existing bound is already very small this would only give a very slight improvement in Theorem \ref{mainthm}.  The main limitation is the restriction to $\eta<\frac{22}{3295}$ rather than the losses in the minorant.  This restriction on $\eta$ was imposed in Lemma \ref{hblem} and it appears that for larger values we cannot avoid an inconveniently large discard at that point of the argument.  In other words, there is a discontinuity in our method at $\eta=\frac{22}{3295}$.

\newpage
\addcontentsline{toc}{section}{References} 
\bibliographystyle{plain}
\bibliography{../biblio}

\vspace*{.5cm}

Roger C. Baker,

Department of Mathematics,

Brigham Young University,

Provo, UT 84602, U. S. A.

{\tt baker@math.byu.edu} 
\bigskip 

Alastair J. Irving

Centre de recherches math\'ematiques,

Universit\'e de Montr\'eal,

Pavillon Andr\'e-Aisenstadt,

2920 Chemin de la tour, Room 5357,

Montr\'eal, (Qu\'ebec), H3T 1J4, Canada

{\tt alastair.j.irving@gmail.com} 
 \end{document}